\documentclass[10pt]{amsart}
\usepackage[centertags]{amsmath}
\usepackage{amsfonts}
\usepackage{amssymb}
\usepackage{amsthm}
\usepackage{graphics,color}
\usepackage{graphicx}
\usepackage{newlfont}

\usepackage{amscd,amsmath,enumerate,verbatim,calc}

\newlength{\defbaselineskip}
\newcommand{\setlinespacing}[1]%
           {\setlength{\baselineskip}{#1 \defbaselineskip}}

\newcommand{\la}{\lambda}

\newcommand{\abs}[1]{\left\vert#1\right\vert}

\def\NI{\noindent}
\def\al{\alpha}
\def\be{\beta}

\def\i{\infty}

\theoremstyle{plain}
\newtheorem{theorem}{Theorem}[section]
\newtheorem{definition}[theorem]{Definition}

\newtheorem{lemma}[theorem]{Lemma}

\newtheorem{conjecture}[theorem]{Conjecture}

\title[Cranks in Ramanujan's Lost Notebook]{A study of the crank function in Ramanujan's Lost Notebook}

\author[M. P. Saikia]{Manjil P. Saikia}

\address{Diploma Student, Mathematics Group, The Abdus Salam International Centre for Theoretical Physics, Strada Costiera-11, I-34151, Miramare, Trieste, Italy}

\email{manjil@gonitsora.com, msaikia@ictp.it}

\begin{document}

\maketitle

\begin{abstract}
In this note, we shall give a brief survey of the results that are found in Ramanujan's Lost Notebook related to cranks. Recent work by B. C. Berndt, H. H. Chan, S. H. Chan and W. -C. Liaw have shown conclusively that cranks was the last mathematical object that Ramanujan studied. We shall closely follow the work of Berndt, Chan, Chan and Liaw and give a brief description of their work.
\end{abstract}

\vskip 3mm


\vskip 3mm

\noindent{\footnotesize 2010 Mathematical Reviews Classification
Numbers: 11-02, 11P81, 11P82, 11P83, 11P84.}

\section{{Introduction}}

This note is divided into three sections. This section is a brief introduction to the results that we discuss in the remainder of the note and to set the notations and preliminaries that we shall be needing in the rest of our study. In the second section, we formally study the crank statistic for the general partition function which is defined as follows.

\begin{definition}[Partition Function]
 If $n$ is a positive integer, let $p(n)$ denote the number of unrestricted representations of $n$ as a sum of positive integers, where representations
 with different orders of the same summands are not regarded as distinct. We call $p(n)$ the partition function. 
\end{definition}

We state a few results of George E. Andrews and Frank G. Garvan \cite{crank} and then show conclusively that Ramanujan had also been studying the crank function much earlier in another guise. We follow the works of Bruce C. Berndt and his collaboraters (\cite{bcb1, bcb2} and \cite{bcb3}) to study the work of Ramnaujan dealing with cranks. In the final section, we give some concluding remarks.

Although no originality is claimed with regards to the results proved here, but in some cases the arguments may have been modified to give an easier justification.

\subsection{Motivation}

The motivation for the note comes from the work of Freeman Dyson \cite{dyson1}, in which he gave combinatorial explanations of the following famous congruences given by Ramanujan

\begin{equation}\label{p1}
 p(5n+4)\equiv 0~(\textup{mod}~5),
\end{equation}

\begin{equation}\label{p2}
 p(7n+5)\equiv 0~(\textup{mod}~7),
\end{equation}

\begin{equation}\label{p3}
 p(11n+6)\equiv 0~(\textup{mod}~11),
\end{equation} where $p(n)$ is the ordinary partition function defined as above.

In order to give combinatorial explanations of the above, Dyson defined the \textit{rank} 
of a partition to be the largest part minus the number of parts. Let $N(m,t,n)$ denote the number of partitions of $n$ with rank congruent to $m$ modulo $t$. Then Dyson conjectured that
$$N(k,5,5n+4)=\frac{p(5n+4)}{5}, ~~0\leq k \leq 4,$$ and $$N(k,7,7n+5)=\frac{p(7n+5)}{7}, ~~0\leq k \leq 6,$$
which yield combinatorial interpretations of \eqref{p1} and \eqref{p2}.

These conjectures were later proven by Atkin and Swinnerton-Dyer in \cite{atkin}. The generating function for $N(m,n)$ is given by

$$\sum_{m=-\infty}^{\infty}\sum_{n=0}^{\infty}N(m,n)a^mq^n=\sum_{n=0}^{\infty}\frac{q^{n^2}}{(aq;q)_n(q/a;q)_n}.$$ Here $\abs q<1, ~\abs q < \abs a <1/\abs q .$ For each nonnegative integer $n$, we set
 
 $$(a)_n:= (a;q)_n:= \prod_{k=0}^{n-1}(1-aq^k), ~~(a)_{\infty}:=(a;q)_{\infty}:=\lim_{n\rightarrow \infty}(a;q)_n, \abs q<1.$$

\NI We also set $$(a_1, \ldots, a_m;q)_n:=(a_1;q)_n\cdots(a_m;q)_n$$ and $$(a_1, \ldots, a_m;q)_{\infty}:=(a_1;q)_{\infty}\cdots(a_m;q)_{\infty}.$$

However, the corresponding analogue of the rank doesn't hold for \eqref{p3}, and so Dyson conjectured the existence of another statistic which he called the \textit{crank}. In his doctoral dissertation F. G. Garvan defined vector partitions which became the forerunner of the true crank. 

Let $\mathcal{P}$ denote the set of partitions and $\mathcal{D}$ denote the set of partitions into distinct parts. Following Garvan, \cite{vcg} we 
denote the set of vector partitions $V$ to be defined by $$V=\mathcal{D}\times \mathcal{P} \times \mathcal{P}.$$ For $\vec{\pi}=(\pi_1, \pi_2, \pi_3)\in V$, we define the weight $\omega(\vec{\pi})=(-1)^{\#(\pi_1)}$, the crank$(\vec{\pi})=\#(\pi_2)-\#(\pi_3)$, and 
$\abs{\vec{\pi}} = \abs{\pi_1} +\abs{\pi_2}+\abs{\pi_3}$, where $\abs{\pi}$ is the sum of the parts of $\pi$. The number of vector partitions of $n$ with crank $m$ counted according to the weight $\omega$ is denoted by $$N_V(m,n)=\sum_{\vec{\pi}\in V, \abs{\vec{\pi}}=n, crank(\vec{\pi})=m}\omega(\vec{\pi}).$$

\NI We then have $$\sum_m N_V(m,n)=p(n)$$ where the summation is taken over all possible cranks $m$.
 
Let $N_V(m,t,n)$ denote the number of vector partitions of $n$ with crank congruent to $m$ modulo $t$ counted according to the weights $\omega$. Then we have
 
 \begin{theorem}[Garvan, \cite{vcg}]
  
  $$N_V(k,5,5n+4)=\frac{p(5n+4)}{5},~0\leq k \leq 4,$$
  $$N_V(k,7,7n+5)=\frac{p(7n+5)}{7},~0\leq k \leq 6,$$
  $$N_V(k,11,11n+6)=\frac{p(11n+6)}{11},~0\leq k \leq 10.$$
  
 \end{theorem}

 On June 6, 1987 at a student dormitory in the University of Illinois, Urbana-Champaign, G. E. Andrews and F. G. Garvan \cite{crank} found the true crank. 

\begin{definition}[Crank]
 For a partition $\pi$, let $\lambda(\pi)$ denote the largest part of $\pi$, let $\mu(\pi)$ denote the number of ones in $\pi$, and let 
 $\nu(\pi)$ denote the number of parts of $\pi$ larger than $\mu(\pi)$. The crank $c(\pi)$ is then defined to be 
 
 $$
c(\pi)=\left\{\begin{array}{cc}
{\lambda(\pi)} & if~\mu(\pi)=0,\\
{\nu(\pi)-\mu(\pi)} & if~\mu(\pi)>0.
\end{array}
\right.
$$
\end{definition}

Let $M(m,n)$ denote the number of partitions of $n$ with crank $m$, and let $M(m,t,n)$ denote the number of partitions of $n$ with crank congruent to 
 $m$ modulo $t$. For $n\leq 1$ we set $M(0,0)=1, M(m,0)=0$, otherwise $M(0,1)=-1, M(1,1)=M(-1,1)=1$ and $M(m,1)=0$ otherwise. The generating function for $M(m,n)$ is given by 
 
 \begin{equation}\label{gfc}
\sum_{m=-\infty}^{\infty}\sum_{n=0}^{\infty}M(m,n)a^mq^n=\frac{(q;q)_{\infty}}{(aq;q)_{\infty}(q/a;q)_{\infty}}.  
 \end{equation}

 The crank not only leads to combinatorial interpretations of \eqref{p1} and \eqref{p2}, but also of \eqref{p3}. In fact we have the following result.

 \begin{theorem}[Andrews-Gravan \cite{crank}]
  With $M(m,t,n)$ defined as above, 
  $$M(k,5,5n+4)=\frac{p(5n+4)}{5},~0\leq k \leq 4,$$
  $$M(k,7,7n+5)=\frac{p(7n+5)}{7},~0\leq k \leq 6,$$
  $$M(k,11,11n+6)=\frac{p(11n+6)}{11},~0\leq k \leq 10.$$
 \end{theorem}

Thus, we see that an observation by Dyson lead to concrete mathematical objects almost 40 years after it was first conjectured. Following the work of Andrews and Garvan, there have been a plethora of results by various authors including Garvan himself, where cranks have been found for many different congruence relations. We see that, the crank turns out to be an interesting object to study.

\subsection{Notations and Preliminaries}

We record here some of the notations and important results that we shall be using throughout our study. 
 Ramanujan's general theta function $f(a,b)$ is defined by  
 $$f(a,b):=\sum_{n=-\infty}^{\infty}a^{n(n+1)/2}b^{n(n-1)/2}, ~\abs{ab}<1.$$

 Our aim in Section 2 is to study the following more general function
 
 $$F_a(q)=\frac{(q;q)_\infty}{(aq;q)_\infty (q/a;q)_\infty}.$$
 
To understand the theorems mentioned above and also the results of Ramnaujan stated in his Lost Notebook and the other results to be discussed here, we need some well-known as well as lesser known results which we shall discuss briefly here. A more detailed discussion can be found in \cite{geasf, qbcb} and \cite{qseries}.

One of the most common tools in dealing with $q$-series identities is the Jacobi Triple Product Identity, given by the following.

\begin{theorem}[Jacobi's Triple Product Identity]
  For $z\neq 0$ and $\abs x<1$, we have $$\prod_{n=0}^{\infty}\{(1+x^{2n+2})(1+x^{2n+1}z)(1+x^{2n+1}z^{-1})\}=\sum_{n=-\infty}^{\infty}x^{n^2}z^n.$$
\end{theorem}

For a very well known proof of this result by Andrews, the reader is referred to \cite{gj}.

It is now easy to verify that Ramanujan's theta function $f(a,b)$ satisfies Jacobi triple product identity

\begin{equation}\label{jtp}
f(a,b)=(-a;ab)_{\infty}(-b;ab)_{\infty}(ab;ab)_{\infty}  
\end{equation}

\noindent and also the elementary identity

\begin{equation}\label{elt1}
f(a,b)=a^{n(n+1)/2}b^{n(n-1)/2}f(a(ab)^n,b(ab)^{-n}),
\end{equation}

\noindent for any integer $n$.

We now state a few results that are used to prove some of the results discussed in the next section.

\begin{theorem}[Ramanujan]\label{ram1}
  If $$A_n:=a^n+a^{-n},$$ then $$\frac{(q;q)_{\i}}{(aq;q)_{\i}(q/a;q)_{\i}}=\frac{1-\sum_{m=1,n=0}^{\i}(-1)^mq^{m(m+1)/2+mn}(A_{n+1}-A_n)}{(q;q)_\i}.$$
\end{theorem}

It is seen that the above theorem is equivalent to the following theorem.

\begin{theorem}[K\u{a}c and Wakimoto]\label{kac}
Let $a_k=(-1)^kq^{k(k+1)/2}$, then $$\frac{(q;q)^2_\i}{(q/x;q)_\i(qx;q)_\i}=\sum_{k=-\i}^\i\frac{a_k(1-x)}{1-xq^k}.$$
\end{theorem}

Several times in the sequel we shall also employ an addition theorem found in Chapter 16 of Ramanujan's second notebook \cite[p.~48,~Entry~31]{nb3}.

\begin{lemma}\label{l2.2}
If $U_n=\al^{n(n+1)/2}\be^{n(n-1)/2}$ and $V_n=\al^{n(n-1)/2}\be^{n(n=1)/2}$ for each integer $n$, then $$f(U_1,V_1)=\sum_{k=0}^{N-1}U_kf\displaystyle\left(\frac{U_{N+k}}{U_k}, \frac{V_{N-k}}{U_k}\displaystyle\right).$$
\end{lemma}

Also useful to us are the following famous results, called the quintuple product identity and Winquist's Identity.

\begin{lemma}[Quintuple Product Identity]\label{qpi}
Let $f(a,b)$ be defined as above, and let $$f(-q):=f(-q,-q^2)=(q;q)_\i,$$ then $$f(P^3Q, Q^4/P^3)-P^2f(Q/P^3, P^3Q^5)=f(-Q^2)\frac{f(-P^2, -Q^2/P^2)}{f(PQ, Q/P)}.$$
\end{lemma}

In \cite{cooper}, Shawn Cooper studied the history and all the known proofs of the above result till 2006 systematically. In particular, mention may be made of Andrews' nifty proof where he uses the $_6\psi_6$ summation formula to derive the above result. It should be mentioned that Sun Kim has
 given the first combinatorial proof of the quintuple product identity in \cite{sunkim}. Zhu Cao  in \cite{cao} has developed a new technique using which he proves the quintuple product identity.

\begin{lemma}[Winquist's Identity]\label{wi}
 Following the notations given earlier, we have 
 \begin{align}
(a, q/a, b, q/b, ab, q/(ab), a/b, bq/a, q, q;q)_{\infty} &= f(-a^3, -q^3/a^3)\{f(-b^3q, -q^2/b^3)-bf(-b^2q^2, -q/b^3)\}\nonumber\\
 & \quad -ab^{-1}f(-b^3, -q^3/b^3)\{f(-a^3q, -q^2/a^3)\nonumber\\
 & \quad -af(-a^3q^2,-q/a^3)\}.\nonumber
 \end{align}
\end{lemma}

Like Lemma \ref{qpi}, this result is also very well known and was first used to prove \eqref{p3} by Winquist. Recently Cao, \cite{cao2} gave a new proof of this result using complex analysis and basic facts about $q$-series.

Now, we have almost all the details that we need for a detailed study of various results related to cranks in the following section.

\section{{Crank for the Partition Function}}

In this section, we shall systematically study the claims related to cranks, that are found in Ramanujan's Lost Notebook. We closely follow the wonderful exposition of Bruce C. Berndt, Heng Huat Chan, Song Heng Chan and Wen- Chin Liaw \cite{bcb1, bcb2} and \cite{bcb3}. We also take the help of the womderful exposition of Andrews and Berndt in \cite{lnb}. This section is an expansion of \cite{mps}.

\subsection{Cranks and Dissections in Ramanujan's Lost Notebook}

As mentioned in Section 1, the generating function for $M(m,n)$ is given by

\begin{equation}\label{gfcc}
\sum_{m=-\infty}^{\infty}\sum_{n=0}^{\infty}M(m,n)a^mq^n=\frac{(q;q)_{\infty}}{(aq;q)_{\infty}(q/a;q)_{\infty}}.  
 \end{equation}
 
Here $|q|<1, ~~|q|<|a|<1/|q|$. Ramanujan has recorded several entries about cranks, mostly about \eqref{gfcc}. At the top of page 179 in his lost notebook \cite{rama}, Ramanujan defines a function $F(q)$ and coefficient $\la_n$, $n\geq 0$ by 

\begin{equation}\label{faq}
F(q):=F_a(q)=\frac{(q;q)_\infty}{(aq;q)_\infty (q/a;q)_\infty}=: \sum_{n=0}^{\infty}\lambda_nq^n.
\end{equation}

Thus, by \eqref{gfcc}, for $n>1$, $$\la_n=\sum_{m=-\i}^\i M(m,n)a^m.$$ Then Ramanujan offers two congruences for $F(q)$. These, like others that are to follow are to be regarded as congruences in the ring of formal power series in the two variables $a$ and $q$. Before giving these congruences, we give the following definition.

\begin{definition}[Dissections]
If $$P(q):=\sum_{n=0}^{\infty}a_nq^n$$ is any power series, then the m-dissection of $P(q)$ is given by $$P(q)=\sum_{j=0}^{m-1}\sum_{n_j=0}^{\infty}a_{n_jm+j}q^{n_jm+j}.$$

For congruences we have $$P(q)\equiv\sum_{j=0}^{m-1}\sum_{n_j=0}^{\infty}a_{n_jm+j,r}q^{n_jm+j}~(\textup{mod}~r)$$ where r is any positive integer and $a_{n_jm+j}\equiv a_{n_jm+j,r}~(\textup{mod}~r)$.
  \end{definition}
We now state the two congruences that were given by Ramanujan below.

\begin{theorem}[$2$-dissection]\label{2dt}
 We have
 \begin{equation}\label{2d}
F_a(\sqrt{q})\equiv \frac{f(-q^3;-q^5)}{(-q^2;q^2)_{\infty}}+\displaystyle\left(a-1+\frac{1}{a}\displaystyle\right)\sqrt{q}\frac{f(-q,-q^7)}{(-q^2;q^2)_{\infty}}~\displaystyle\left(\textup{mod}~a^2+\frac{1}{a^2}\displaystyle\right). 
 \end{equation}
 \end{theorem}
 
  We note that $\lambda_2=a^2+a^{-2}$, which trivially implies that $a^4\equiv -1~(\textup{mod}~\lambda_2)$ and $a^8\equiv 1~(\textup{mod}~\lambda_2)$. Thus, in \eqref{2d} 
 $a$ behaves like a primitive 8th root of unity modulo $\lambda_2$. If we let $a=\textup{exp}(2\pi i/8)$ and replace $q$ by $q^2$ in the definition of a dissection, \eqref{2d} will give the $2$-dissection of $F_a(q)$.

\begin{theorem}[$3$-dissection]\label{3dt}
We have
\begin{align}
 F_q(q^{1/3}) & \equiv \frac{f(-q^2,-q^7)f(-q^4,-q^5)}{(q^9;q^9)_{\infty}}+\displaystyle\left(a-1+1/a\displaystyle\right)q^{1/3}\frac{f(-q,-q^8)f(-q^4,-q^5)}{(q^9;q^9)_{\infty}}\nonumber\\
 & \quad +\displaystyle\left(a^2+\frac{1}{a^2}\displaystyle\right)q^{2/3}\frac{f(-q,-q^8)f(-q^2,-q^7)}{(q^9;q^9)_{\infty}}~\displaystyle\left(\textup{mod}~a^3+1+\frac{1}{a^3}\displaystyle\right).\label{3d}  
\end{align}
\end{theorem}

Again we note that, $\lambda_3=a^3+1+\frac{1}{a^3}$, from which it follows that $a^9\equiv -a^6-a^3\equiv 1~(\textup{mod}~\lambda_3)$. So in \eqref{3d}, $a$ behaves 
 like a primitive 9th root of unity. While if we let $a=\textup{exp}(2\pi i/9)$ and replace $q$ by $q^3$ in the definition of a dissection, \eqref{3d} will give the $3$-dissection of $F_a(q)$.

In contrast to \eqref{2d} and \eqref{3d}, Ramanujan offered the $5$-dissection in terms of an equality. 

\begin{theorem}[$5$-dissection]\label{5dt}
We have
 \begin{align}
 F_a(q) &= \frac{f(-q^2,-q^3)}{f^2(-q,-q^4)}f^2(-q^5)-4\cos^2(2n\pi/5)q^{1/5}\frac{f^2(-q^5)}{f(-q,-q^4)}\\ \nonumber
 &\quad +2\cos(4n\pi/5)q^{2/5}\frac{f^2(-q^5)}{f(-q^2,-q^3)}-2\cos(2n\pi/5)q^{3/5}\frac{f(-q,-q^4)}{f^2(-q^2,-q^3)}f^2(-q^5).\label{5d}
 \end{align}

\end{theorem}

 We observe that (2.13) has no term with $q^{4/5}$, which is a reflection of \eqref{p1}. In fact, one can replace (2.13) by a congruence and in turn \eqref{2d} and \eqref{3d} by equalities. This is done in \cite{bcb1}. Ramanujan did not specifically give the $7$- and $11$-dissections of $F_a(q)$ in \cite{rama}.  However, he vaugely gives some of the coefficients occuring in those dissections. Uniform proofs of these dissections and the others already stated earlier are given in \cite{bcb1}.
 
 Ramanujan gives the $5$-dissection of $F(q)$ on page 20 of his lost notebook \cite{rama}. It is interesting to note that he does not give the alternate form analogous to those of \eqref{2d} amd \eqref{3d}, from which the $5$-dissection will follow if we set $a$ to be a primitive fifth root of unity. On page 59 in his lost notebook \cite{rama}, Ramanujan has recorded a quotient of two power series, with the highest power of the numerator being $q^{21}$ and the highest power of the denominator being $q^{22}$. Underneath he records another power series with the highest power being $q^5$. Although not claimed by him, the two expressions are equal. This claim was stated in the previous section as Theorem \ref{ram1}.
 
In the following, we shall use the results and notations that we discussed in Section 1. It must be noted that some of these results have been proved by numerous authors, but we follow the exposition of Berndt, Chan, Chan and Liaw \cite{bcb1}.
\begin{theorem}
We have
\begin{equation}\label{2da}
F_a(q)\equiv \frac{f(-q^6,-q^{10})}{(-q^4;q^4)_\i}+\displaystyle\left(a-1+\frac{1}{a}\displaystyle\right)q\frac{f(-q^2,-q^{14})}{(-q^4;q^4)_\i}~(\textup{mod}~A_2).
\end{equation}
\end{theorem}

We note that \eqref{2da} is equivalent to \eqref{2d} if we replace $\sqrt q$ by $q$. We shall give a proof of this result using a method of \textit{rationalization}. This method does not work in general, but only for those $n$-dissections where $n$ is \textit{small}.

\begin{proof}
Throughout the proof, we assume that $|  q|  <|  a |  <1/|  q| $ and we shall also frequently use the facts that $a^4\equiv -1$ modulo $A_2$ and that $a^8\equiv 1$ modulo $A_2$. We write
\begin{equation}\label{3.2}
\frac{(q;q)_\i}{(aq;q)_\i(q/a;q)_\i}=(q;q)_\i\prod_{n=1}^\i\displaystyle\left(\sum_{k=0}^\i(aq^n)^k\displaystyle\right)\displaystyle\left(\sum_{k=0}^\i(q^n/a)^k\displaystyle\right).
\end{equation} We now subdivide the series under product sign into residue classes modulo $8$ and then sum the series. Using repeatedly congruences modulo $8$ for the powers of $a$, we shall obtain from \eqref{3.2} the following
\begin{equation}\label{3.3}
\frac{(q;q)_\i}{(aq;q)_\i(q/a;q)_\i}\equiv \frac{(q;q)_\i}{(-q^4;q^4)_\i}\prod_{n=1}^\i(1+aq^n)(1+q^n/a)~(\textup{mod}~A_2),
\end{equation}upon multiplying out the polynomials in the product that we obtain and using the congruences for powers of $a$ modulo $A_2$.

Now, using Lemma \ref{l2.2} with $\al=a, \be=q/a$ amd congruences for powers of $a$ modulo $A_2$, we shall find after some simple manipulations $$(q;q)_\i(-aq;q)_\i(-q/a;q)_\i\equiv f(-q^6,-q^{10})+(A_1-1)qf(-q^2,-q^{14})~(\textup{mod}~A_2).$$ Using \eqref{3.3} in the above, we shall get the desired result.
\end{proof}

Using similar techniques, but with more complicated manipulations, we shall be able to find the following theorem.

\begin{theorem}\label{3da}
We have
\begin{align}
 F_a(q) & \equiv \frac{f(-q^6,-q^{21})f(-q^{12},-q^{15})}{(q^{27};q^{27})_{\infty}}+(A_1-1)q\frac{f(-q^3,-q^{24})f(-q^{12},-q^{15})}{(q^{27};q^{27})_{\infty}}\nonumber\\
 & \quad +A_2q^3\frac{f(-q^3,-q^{24})f(-q^{6},-q^{21})}{(q^{27};q^{27})_{\infty}}~(\textup{mod}~A_3+1).\nonumber
\end{align}
\end{theorem}

When $a=e^{2\pi i/9}$, the above theorem gives us Theorem \ref{3dt}. For the remainder of this section, we shall use the following notation
\begin{equation}\label{5.1}
S_n(a):=\sum_{k=-n}^n a^k.
\end{equation}
We note that, when $p$ is an odd prime then $$S_{(p-1)/2}(a)=a^{(1-p)/2}\Phi_p(a),$$ where $\Phi_n(a)$ is the minimal, monic polynomial for a primitive $n$th root of unity. We now give the $5$-dissection in terms of a congruence.

\begin{theorem}\label{5da}
With $f(-q)$, $S_2$ and $A_n$ as defined earlier, we have
\begin{align}
F_a(q) & \equiv \frac{f(-q^{10},-q^{15})}{f^2(-q^5,-q^{20})}f^2(-q^{25})+(A_1-1)q\frac{f^2(-q^{25})}{f(-q^5,-q^{20})}\nonumber\\
& \quad +A_2q^2\frac{f^2(-q^{25})}{f(-q^{10},-q^{15})}-A_1q^3\frac{f(-q^5, -q^{20})}{f^2(-q^{10},-q^{15})}f^2(-q^{25})~(\textup{mod}~S_2).\nonumber
\end{align}
\end{theorem}

In his lost notebook \cite[p.~58, 59, 182]{rama}, Ramanujan factored the coefficients of $F_a(q)$ as functions of $a$. In particular, he sought factors of $S_2$ in the coefficients. The proof of Theorem \ref{5da} uses the famous Rogers-Ramanujan continued fraction identities, which we shall omit here.

We now state and prove the following $7$-dissection of $F_a(q)$.

\begin{theorem}\label{7da}
With usual notations defined earlier, we have
\begin{align}
\frac{(q;q)_\i}{(qa;q)_\i(q/a;q)_\i} & \equiv \frac{1}{f(-q^7)}(A^2+(A-1-1)qAB+A_2q^2B^3+(A_3+1)q^3AC\nonumber\\
& \quad -A_1q^4BC-(A_2+1)q^6C^2)~(\textup{mod}~S_3),\nonumber
\end{align} where $A=f(-q^{21}, -q^{28})$, $B=f(-q^{35}, -q^{14})$ and $C=f(-q^{42}, -q^7)$.
\end{theorem}

\begin{proof}
Rationalizing and using Theorem \ref{jtp} we find
\begin{align}
\frac{(q;q)_\i}{(qa;q)_\i(q/a;q)_\i} & \equiv \frac{(q;q)^2_\i(qa^2;q)_\i(q/a^2;q)_\i(qa^3;q)_\i(q/a^3;q)_\i}{(q^7;q^7)_\i}\label{6.2}\\
& \equiv \frac{1}{f(-q^7)}\frac{f(-a^2,-q/a^2)}{(1-a^2)}\frac{f(-a^3,-q/a^3}{(1-a^3)}~(\textup{mod}~S_3).\nonumber
\end{align}
Now using Lemma \ref{l2.2} with $(\al, \be, N)=(-a^2, -q/a^2, 7)$ and $(-a^3, -q/a^3, 7)$ respectively, we find that
\begin{equation}\label{6.3}
\frac{f(-a^2,-q/a^2)}{(1-a^2)}\equiv A-q\frac{(a^5-a^4)}{(1-a^2)}B+q^3\frac{(a^3-a^6)}{(1-a^2}C~(\textup{mod}~S_3)
\end{equation} and 
\begin{equation}\label{6.4}
\frac{f(-a^3,-q/a^3)}{(1-a^3)}\equiv A-q\frac{(a^4-a^6)}{(1-a^3)}B+q^3\frac{(a-a^2)}{(1-a^3}C~(\textup{mod}~S_3).
\end{equation} Now substituting \eqref{6.3} and \eqref{6.4} in \eqref{6.2} and simplifying, we shall finish the proof of the desired result.
\end{proof}

Although there is a result for the 11 dissection as well, we do not discuss it here. We just state the $11$-disection of $F_a(q)$ below.

\begin{theorem}\label{11da}
With $A_m$ and $S_5$ defined as earlier, we have
\begin{align}
F_a(q) & \equiv \frac{1}{(q^{11};q^{11})_\i(q^{121};q^{121})^2_\i}(ABCD+\{A_1-1\}qA^2BE\nonumber\\
& \quad +A_2q^2AC^2D+\{A_3+1\}q^3ABD^2\nonumber\\
& \quad +\{A_2+A_4+1\}q^4ABCE-\{A_2+A_4\}q^5B^2CE\nonumber\\
& \quad +\{A_1+A_4\}q^7ABDE-\{A_2+A_5+1\}q^{19}CDE^2\nonumber\\
& \quad -\{A_4+1\}q^9ACDE-\{A_3\}q^{10}BCDE)~(\textup{mod}~S_5),\nonumber
\end{align}
where $A=f(-q^{55}, -q^{66})$, $B=f(-q^{77}, -q^{44})$, $C=f(-q^{88}, -q^{33})$, $D=f(-q^{99}, -q^{22}),$ and $E=f(-q^{110}, -q^{11})$.
\end{theorem}

Although the proof of Theorem \ref{11da} is not difficult, it is very tedious involving various routine $q$-series identities and results mentioned in Section 1, so we shall omit it here.

\subsection{Other results from the Lost Notebook}

Apart from the results that we have discussed so far in this section, Ramanujan also recorded many more entires is his lost notebook which pertain to cranks. For example, on page 58 in his lost notebook \cite{rama}, Ramanujan has written out the first $21$ coefficients in the power series representation of the crank $F_a(q)$, where he had incorrectly written the coefficient of $q^{21}$. On the following page, beginning with the coefficient of $q^{13}$, Ramanujan listed some (but not necessarily all) of the factors of the coefficients up to $q^{26}$. He did not indicate why he recorded an incomplete list of such factors. However, it can be noted that in each case he recorded linear factors only when the leading index is $\leq 5$.

On pages 179 and 180 in his lost notebook \cite{rama}, Ramanujan offered ten tables of indices of coefficients $\la_n$ satisfying certain congruences. On page 61 in \cite{rama}, he offers rougher drafts of nine of these ten tables, where Table 6 is missing. Unlike the tables on pages 179 and 180, no explanations are given on page 61. It is clear that Ramanujan had calculated factors beyond those he has recorded on pages 58 and 59 of his lost notebook as mentioned in the earlier paragraph. In \cite{bcb2}, the authors have verified these claims using a computer algebra software. Among other results, O.--Y. Chan \cite{oyc} has verified all the tables of Ramanujan. We explain below these tables following Berndt, Chan, Chan and Liaw \cite{bcb2}.

\begin{center}
\textbf{Table 1.} $\la_n\equiv 0~(\textup{mod}~a^2+1/a^2)$
\end{center}

Here Ramanujan indicates which coefficients of $\la_n$ have $a_2$ as a factor. If we replace $q$ by $q^2$ in \eqref{2d}, we see that Table 1 contains the degree of $q$ for those terms with zero coefficients for both $\dfrac{f(-q^6, -q^{10})}{(-q^4;q^4)_\i}$ and $q\dfrac{f(-q^2, -q^{14})}{(-q^4;q^4)_\i}$. There are $47$ such values.

\begin{center}
\textbf{Table 2.} $\la_n\equiv 1~(\textup{mod}~a^2+1/a^2)$
\end{center}

Returning again to \eqref{2d} and replacing $q$ by $q^2$, we see that Ramanujan recorded all the indices of the coefficients that are equal to $1$ in the power series expansion of $\dfrac{f(-q^6, -q^{10})}{(-q^4;q^4)_\i}$. There are $27$ such values.

\begin{center}
\textbf{Table 3.} $\la_n\equiv -1~(\textup{mod}~a^2+1/a^2)$
\end{center}

This table can be understood in a similar way as the previous table. There are $27$ such values.

\begin{center}
\textbf{Table 4.} $\la_n\equiv a-1+\frac{1}{a}~(\textup{mod}~a^2+1/a^2)$
\end{center}

Again looking at \eqref{2d}, we note that $a-1+\frac{1}{a}$ occurs as a factor of the second expression on the right side. Thus replacing $q$ by $q^2$, Ramanujan records the indices of all coefficients of $q\dfrac{f(-q^2, -q^{14})}{(-q^4;q^4)_\i}$ which are equal to $1$. There are $22$ such values.

\begin{center}
\textbf{Table 5.} $\la_n\equiv -\displaystyle\left(a-1+\frac{1}{a}\displaystyle\right)~(\textup{mod}~a^2+1/a^2)$
\end{center}

This table can also be interpreted in a manner similar to the previous one. There are $23$ such values.

\begin{center}
\textbf{Table 6.} $\la_n\equiv 0~(\textup{mod}~a+1/a)$
\end{center}

Ramanujan here gives those coefficients which have $a_1$ as a factor. There are only three such values and these values can be discerned from the table on page 59 of the lost notebook.

From the calculation $$\frac{(q;q)_\i}{(aq;q)_\i)(q/a;q)_\i}\equiv \frac{(q;q)_\i}{(-q^2;q^2)_\i} = \frac{f(-q)f(-q^2)}{f(-q^4)}~(\textup{mod}~a+1/a),$$ we see that in this table Ramanujan has recorded the degree of $q$ for the terms with zero coefficients in the power series expansion of $\dfrac{f(-q)f(-q^2)}{f(-q^4)}$. 

From the next three tables, it is clear from the calculation $$\frac{(q;q)_\i}{(aq;q)_\i)(q/a;q)_\i}\equiv \frac{(q^2;q^2)_\i}{(-q^3;q^3)_\i} = \frac{f(-q^2)f(-q^3)}{f(-q^6)}~(\textup{mod}~a+1/a),$$ that Ramanujan recorded the degree of $q$ for the terms with coefficients $0$, $1$ and $-1$ respectively in the power series expansion of $\dfrac{f(-q^2)f(-q^3)}{f(-q^6)}$.

\begin{center}
\textbf{Table 7.} $\la_n\equiv 0~(\textup{mod}~a-1+1/a)$
\end{center}

There are $19$ such values.

\begin{center}
\textbf{Table 8.} $\la_n\equiv 1~(\textup{mod}~a-1+1/a)$
\end{center}

There are $26$ such values.

\begin{center}
\textbf{Table 9.} $\la_n\equiv -1~(\textup{mod}~a-1+1/a)$
\end{center}

There are $26$ such values.

\begin{center}
\textbf{Table 10.} $\la_n\equiv 0~(\textup{mod}~a+1+1/a)$
\end{center}

Ramanujan put $2$ such values. From the calculation $$\frac{(q;q)_\i}{(aq;q)_\i)(q/a;q)_\i}\equiv \frac{(q;q)^2_\i}{(-q^3;q^3)_\i} = \frac{f^2(-q)}{f(-q^3)}~(\textup{mod}~a+1+1/a),$$ it is clear that Ramanujan had recorded the degree of $q$ for the terms with zero coefficients in the power series expansion of $\dfrac{f^2(-q)}{f(-q^3)}$.

The infinite products  $\dfrac{f(-q^6, -q^{10})}{(-q^4;q^4)_\i}$, $q\dfrac{f(-q^2, -q^{14})}{(-q^4;q^4)_\i}$, $\dfrac{f(-q)f(-q^2)}{f(-q^4)}$, $\dfrac{f(-q^2)f(-q^3)}{f(-q^6)}$ and $\dfrac{f^2(-q)}{f(-q^3)}$ do not appear to have monotonic coefficients for sufficiently large $n$. However, if these products are dissected, then we have the following conjectures by Berndt, Chan, Chan and Liaw \cite{bcb2}.

\begin{conjecture}\label{cj1}
Each component in each of the dissections for the five products given above has monotonic coefficients for powers of $q$ above $600$.
\end{conjecture}

The authors have checked this conjecture for $n=2000$.

\begin{conjecture}\label{cj2}
For any positive integers $\al$ and $\be$, each component of the $(\al+\be+1)$-dissection of the product $$\frac{f(-q^\al)f(-q^\be)}{f(-q^{al+\be+1})}$$ has monotonic coefficients for sufficiently large powers of $q$.
\end{conjecture}

It is clear that Conjecture \ref{cj1} is a special case of Conjecture \ref{cj2} for the last three infinite products given above when we set $(\al, \be)=(1,2), (2,3)$, and $(1,1)$ respectively. Using the Hardy-Ramanujan circle method, these conjectures have been verified by O. -Y. Chan \cite{oyc}.

On page 182 in his lost notebook \cite{rama}, Ramanujan returns to the coefficients $\la_n$ in \eqref{gfcc}. He factors $\la_n$ for $1\leq n\leq n$ as before, but singles out nine particular factors by giving them special notation. These are the factors which occured more than once. Ramanujan uses these factors to compute $p(n)$ which is a special case of \eqref{gfcc} with $a=1$. It is possible that through this, Ramanujan may have been searching for results through which he would have been able to give some divisibility criterion of $p(n)$. Ramanujan has left no results related to these factors, and it is up to speculation as to his motives for doing this.

Again on page 59, Ramanujan lists two factors, one of which is Theorem \ref{ram1}. further below this he records two series, namely,

\begin{equation}\label{8.6}
S_1(a,q):=\frac{1}{1+a}+\sum_{n=1}^\i\displaystyle\left(\frac{(-1)^nq^{n(n+1)/2}}{1+aq^n}+\frac{(-1)^nq^{n(n+1)/2}}{a+q^n}\displaystyle\right)
\end{equation}

\noindent and

\begin{equation}\label{8.7}
S_2(a,q):=1+\sum_{m=1, n=0}^{\i}(-1)^{m+n}q^{m(m+1)/2+nm}(a_{n+1}+a_n),
\end{equation}

\noindent where here $a_0:=1$. Although no result has been written by Ramanujan, however the authors in \cite{bcb2} have found the following theorem.

\begin{theorem}
With the notations as described above we have $$(1+a)S_1(a,q)=S_2(a,q)=F_{-a}(q).$$
\end{theorem}

\begin{proof}
We multiply \eqref{8.6} by $(1+a)$ to get
\begin{align}
(1+a)S_1(a,q) &= 1+(1+a)\sum_{n=1}^\i\displaystyle\left(\frac{(-1)^nq^{n(n+1)/2}}{1+aq^n}+\frac{(-1)^nq^{n(n+1)/2}}{a+q^n}\displaystyle\right)\nonumber\\
&= 1+(1+a)\sum_{n=1}^\i\displaystyle\left(\frac{(-1)^nq^{n(n+1)/2}}{1+aq^n}+\frac{(-1)^nq^{n(n-1)/2}}{1+aq^{-n}}\displaystyle\right)\nonumber\\
&= 1+(1+a)\sum_{n\neq 0}\frac{(-1)^nq^{n(n+1)/2}}{1+aq^n}\nonumber\\
&= \sum_{n=-\i}^\i\frac{(-1)^nq^{n(n+1)/2}(1+a)}{1+aq^n}\nonumber\\
&= \frac{(q;q)^2_\i}{(-aq;q)_\i(-q/a;q)_\i},\label{9.8}
\end{align} by Theorem \ref{kac}.

Secondly by Theorem \ref{ram1}, we have
\begin{align}
S_2(a,q) &= 1+\sum_{m=1, n=0}^\i(-1)^mq^{m(m+1)/2+nm}(-(-a)^{n+1}-(-a)^{-n-1}+(-a)^n+(-a)^{-n})\nonumber\\
&= \frac{(q;q)^2_\i}{(-aq;q)_\i(-q/a;q)_\i}.\label{9.9}
\end{align} Thus, \eqref{9.8} and \eqref{9.9} completes the proof.

\end{proof}

From the preceeding discussion, it is clear that Ramanujan was very interested in finding some general results with the possible intention of determining arithmetical properties of $p(n)$ from them by setting $a=1$. Although he found many beautiful results, but his goal eluded him. The kind of general theorems on the divisibility of $\la_n$ by sums of powers of $a$ appear to be very difficult. Also, a challenging problem is to show that Ramanujan's Table 6 is complete.

\subsection{Cranks -- The final problem}

In his last letter to G. H. Hardy, Ramanujan announced a new class of functions which he called mock theta functions, and gave several examples and theorems related to them. The latter was dated 20th January, 1920, a little more than three months before his death. Thus for a long time, it was widely believed that the last problem on which Ramanujan worked on was mock theta functions. But the wide range of topics that are covered in his lost notebook \cite{rama} suggests that he had worked on several problems in his death bed. Of course, it is only speculation and some educated guess that we can make. But, the work of Berndt, Chan, Chan and Liaw \cite{bcb3} has provided evidence that the last problem on which Ramanujan worked on was cranks, although he would not have used this terminology.

We have already seen, the various dissections of the crank generating function that Ramanujan had provided. In the preceeding section, we saw various other results of Ramanujan related to the coefficients $\la_n$. In order to calculate those tables, Ramanujan had to calculate with hand various series upto hundreds of coefficients. Even for Ramanujan, this is a tremendous task and we can only wonder what might have led him to such a task. Clearly, there was very little chance that he might have found some nice congruences for these values like \eqref{p1} for example.

In the foregoing discussion of the crank, pages 20, 59--59, 61, 53--64, 70--71, and 179--181 are cited from \cite{rama}. Ramanujan's $5$-dissection for the crank is given on page 20. The remaining ten pages are devoted to the crank.  In fact in pages 58--89, there is some scratch work from which it is very difficult to see where Ramanujan was aiming them at. But it is likely that all these pages were related to the crank. In \cite{bcb3}, the authors have remarked that pages 65 (same as page 73), 66, 72, 77, 80--81, and 83--85 are almost surely related to the crank, while they were unable to determine conclusively if the remaining pages pertain to cranks.

In 1983, Ramanujan's widow Janaki told Berndt that there were more pages of Ramanujan's work than the 138 pages of the lost notebook. She claimed that during her husband's funeral service, some gentlemen came and took away some of her husband's papers. The remaining papers of Ramanujan were donated to the University of Madras. It is possible that Ramanujan had two stacks of paper, one for scratch work or work which he did not think complete, and the other where he put down the results in a more complete form. The pages that we have analysed most certainly belonged to the first stack and it was Ramanujan's intention to return to them later. In the time before his death, it is certainly clear that most of Ramanujan's mathematical thoughts had been only on one topic - cranks.

However, one thing is clear that Ramanujan probably didn't think of the crank or rank as we think of it now. We do not know with certainty whether or not Ramanujan thought combinatorially about the crank. Since his notebooks contain very little words and also since he was in his death bed so he didn't waste his time on definitions and observations which might have been obvious to him; so we are not certain about the extent to which Ramanujan thought about these objects. It is clear from some of his published papers that Ramanujan was an excellent combinatorial thinker and it would not be surprising if he had many combinatorial insights about the crank. But, since there is not much recorded history from this period of his life, we can at best only speculate.

\section{{Concluding Remarks}}

Although we have focused here on only one aspect of Ramanujan's work related to crank, this is by no means the complete picture. Recent work by many distinguished mathematicians have shed light on many different aspects of Ramanujan's work related to cranks. We plan to address those issues in a subsequent article. The interested reader wanting to know more about Ramanujan's mathematics can look at \cite{spirit}. For more information on cranks and its story, a good place is \cite{dyson2}. For some more work of Garvan related to cranks, the interested reader can look into \cite{vcg} and \cite{garvan}. There are some more results in the Lost Notebook related to cranks, Andrews and Berndt has provided an excellent exposition of those in Chapters 2, 3 and 4 of \cite{lnb}.

\section*{{Acknowledgements}}

This note is a part of the Masters Thesis of the author \cite{mps2} submitted to Tezpur University in 2014. The author expresses his gratitude to Prof. Nayandeep Deka Baruah for supervising the thesis and for encouraging the author in various stages of his education. The author also acknowledges the helpful remarks received from Prof. Bruce C. Berndt, Dr. Atul Dixit and Mr. Zakir Ahmed related to Ramanujan's mathematics. The author is also thankful to Dr. Shawn Cooper, the anonymous referee and the editor for their helpful remarks.

\end{document}